\documentclass{amsart}

\usepackage{amsmath, amssymb,epic,graphicx,mathrsfs,enumerate}
\usepackage[all]{xy}

\usepackage{latexsym}
\usepackage{longtable}
\usepackage{epsfig}
\usepackage{hhline}
\usepackage[usenames,dvipsnames]{color}

\newtheorem{proposition}{Proposition}[section]
\newtheorem{definition}[proposition]{Definition}

\newtheorem{lemma}[proposition]{Lemma}

\newtheorem{theorem}[proposition]{Theorem}

\newtheorem{remark}[proposition]{Remark}

\begin{document}
\title{On the orbital diameter of classical groups in standard actions}

\author{Attila Mar\'oti}
\address{Hun-Ren Alfr\'ed R\'enyi Institute of Mathematics, Re\'altanoda utca 13-15, H-1053, Budapest, Hungary}
\email{maroti.attila@renyi.hu}

\author{Kamilla Rekv\'enyi} 
\address{Department of Mathematics, University of Manchester, Manchester, M13 9PL, United Kingdom. Also affiliated with: Heilbronn Institute for Mathematical Research, Bristol, BS8 1UG, United Kingdom.}
\email{kamilla.rekvenyi@manchester.ac.uk}

\thanks{The first author was supported by the National Research, Development and Innovation Office (NKFIH) Grant No.~K153681 and Grant No.~K138828.}
\maketitle

\begin{abstract}
Let $G$ be a primitive permutation group acting on a finite set $X$. The orbital diameter $\mathrm{diam}(X,G)$ is defined to be the supremum of the diameters of the (connected) orbital graphs of $G$ after disregarding the directions of all edges in the graphs. This invariant is studied in the case when $G$ is an almost simple group in a standard action. A lower bound is given for $\mathrm{diam}(X,G)$ and we provide a partial classification of pairs $(X,G)$ for which the orbital diameter is at most $2$.    
\end{abstract}

\section{Introduction}

Let $G$ be a permutation group acting on a set $X$. An {\it orbital} is an orbit of $G$ on $X \times X$. The {\it orbital graph} associated with an orbital $E$ is the directed graph with vertex set $X$ and edge set $E$. If $G$ is transitive and $E$ consists of loops, then the orbital graph is called {\it diagonal}. The criterion of Higman \cite{Higman} (see also \cite[Theorem 1.9]{Cameron}) states that a transitive permutation group is primitive if and only if all non-diagonal orbital graphs are connected. (A directed graph is said to be {\it connected} if the associated undirected graph is connected.) In this paper we suppose that $G$ is a finite transitive permutation group (and $X$ is a finite set). In this case a connected orbital graph is strongly connected (see \cite[Theorem 1.10]{Cameron}). (A directed graph is said to be {\it strongly connected} if every vertex may be reached from any other vertex along a path with directed edges.)    

In this paper an (undirected) orbital graph for $(X,G)$ is a graph with vertex set $X$ whose edge set is an orbit of $G$ on the collection of unordered $2$-element subsets of $X$. When $G$ is primitive on a finite set $X$, we shall write $\mathrm{diam}(X,G)$ for the supremum of the diameters of the undirected, non-diagonal orbital graphs for $(X,G)$ and call it the {\it orbital diameter} of $G$ acting on $X$. We will focus on the orbital diameter of classical groups in their standard actions.

Liebeck, Macpherson and Tent \cite{LMT} classified the infinite families of primitive permutation groups with bounded orbital diameter. These results and their proof methods were motivated by model theory and hence \cite{LMT} contains no explicit bounds. Since then, the orbital diameter has been extensively studied from a purely group theoretical standpoint. Some explicit bounds in the case of almost simple groups with alternating socles are provided by Sheikh \cite{Sheikh}. Such groups with orbital diameter at most $5$ are also described. Similarly, explicit bounds and descriptions of groups with small orbital diameters are given by the second author \cite{kamilla} for the case of primitive groups of simple diagonal type and in \cite{kamilla2} for primitive affine groups. For upper bounds in the affine case, see \cite{SKM1}, \cite{SKM2} and \cite{SK}.

For an integer $n \geq 2$ and a prime power $q$, let $\mathrm{Cl}_{n}(q)$ denote any of the groups $\mathrm{PSL}_{n}(q)$, $\mathrm{PSp}_{n}(q)$, $\mathrm{PSU}_{n}(q)$, $\mathrm{P\Omega}^{\pm}_{n}(q)$. Let $G$ be an almost simple primitive permutation group acting on a set $X$ of size $m$. Let the socle of $G$ be isomorphic to $\mathrm{A}_{n}$ or $\mathrm{Cl}_{n}(q)$. Let $t$ be a positive integer at most $n-1$. We now introduce the definition of a standard $t$-action taken from \cite[Section 1.2]{LMT}.   

\begin{definition}
\label{d1}	
Let $t$ be a positive integer and let $X$ be a finite set. Let $G$ be an almost simple primitive permutation group acting on $X$. Let the socle of $G$ be $G_{0}$ (a non-abelian simple group). We say that the group $G$ acting on $X$ has a {\it standard $t$-action} if any of the following holds.
	
	\begin{enumerate}
		
		\item[(a)] $G_{0} = \mathrm{A}_{n}$ and $X = I^{\{ t \}}$, the set of $t$-subsets of $I = \{ 1, \ldots, n \}$ with the natural action of $\mathrm{A}_n$.
		
		\item[(b)] $G_{0} = \mathrm{Cl}_{n}(q)$ and $X$ is an orbit of subspaces of dimension or codimension $t$ in the natural module $V_{n}(q)$; the subspaces are arbitrary if $G_{0} = \mathrm{PSL}_{n}(q)$, and otherwise are totally singular (here we denote $X$ by $\mathcal{S}_t$), non-degenerate (here we denote $X$ by $\mathcal{N}_t$), or, if $G_{0}$ is orthogonal and $q$ is even, are non-singular $1$-spaces (in which case $t=1$ and $X$ is denoted by $\mathcal{N}_1$). 
		
		\item[(c)] $G_{0} = \mathrm{PSL}_{n}(q)$, $G$ contains a graph automorphism of $G_{0}$, and $X$ is an orbit of pairs of subspaces $\{ U, W \}$ of $V = V_{n}(q)$, where either $U \subseteq W$ or $V = U \oplus W$, and $\dim U = t$, $\dim W = n-t$.
		
		\item[(d)] $G_{0} = \mathrm{Sp}_{n}(q)$, $q$ even, and a point stabilizer in $G_{0}$ is $\mathrm{O}^{\pm}_{n}(q)$ (here we take $t = 1$).
	\end{enumerate}
\end{definition}
\begin{remark}\rm
    Note that if $G_0$ is orthogonal and $t$ is even, then $\mathcal{N}_t$ has two types which we will denote by $\mathrm{O}_t^+$ and $\mathrm{O}_t^-.$ 
\end{remark}

Let $k = \min \{ t, \ n-t \}$. Note that in \cite[p. 229]{LMT} it is remarked that it is easy to establish the bound $\mathrm{diam}(X,G) \geq k$ in case $G_{0} = \mathrm{PSL}_{n}(q)$. In this paper we provide a proof for this fact. In fact we show that in case (b) for $G_{0} = \mathrm{PSL}_{n}(q),$ $\mathrm{diam}(X,G) =k.$ In \cite[p. 230]{LMT} it is proved that if $G_{0}$ is a classical group different from $\mathrm{PSL}_{n}(q)$ and $\mathrm{diam}(X,G)$ is bounded, then $k$ is bounded.

Our first main result is the following.

\begin{theorem}\label{thm1}
Let $G$ be a primitive permutation group acting on a finite set $X$ such that $G$ has a standard $t$-action for some positive integer $t$ as in (a)-(d) of Definition \ref{d1}. Put $k = \min \{ t, n-t \}$. If $(G_0,X)\neq (\mathrm{P\Omega}_n^{+}(q),\mathcal{S}_{n/2})$, then $\mathrm{diam}(X,G)\geq k$, otherwise $\mathrm{diam}(X,G) = \lfloor k/2\rfloor$. 
\end{theorem}

We now turn to the classification of primitive permutation groups $G$ acting on finite sets $X$ in standard $t$-actions for integers $t$ such that $\mathrm{diam}(X,G) \leq 2$. 

Let $G$ be a primitive permutation group acting on a finite set $X$ such that $\mathrm{diam}(X,G) = 1$. In this case $G$ is a $2$-homogeneous permutation group. It follows that $G$ is $2$-transitive or is of odd order. Let $G$ have a standard $t$-action for some positive integer $t$. Since $G$ has even order, $G$ must be a $2$-transitive permutation group. The finite $2$-transitive almost simple groups have been classified in \cite[Theorem 5.3]{Cameroncikk} (see also Note 2 after \cite[Theorem 5.3]{Cameroncikk}). The $2$-transitive groups $G$ in standard $t$-actions are the groups in (a) of Definition \ref{d1} with $t \in \{ 1, n-1 \}$, the groups in (b) of Definition \ref{d1} with $G_{0} = \mathrm{PSL}_{n}(q)$ and $t \in \{ 1, n-1 \}$, the groups in (b) of Definition \ref{d1} with $G_{0} = \mathrm{PSU}_{3}(q)$ ($q > 2$) and $X = \mathcal{S}_t$, and the groups in (d) of Definition \ref{d1} (in the latter case see also \cite[Theorem 2]{inglis}). 

Now let $G$ be a primitive permutation group acting on a finite set $X$ such that $\mathrm{diam}(X,G) = 2$. Suppose that $G$ has a standard $t$-action for $X$ for some positive integer $t$. The pairs $(X,G)$ satisfying (a) of Definition \ref{d1} are completely described in \cite[Theorem 1.3 (1)]{Sheikh}. The pair $(X,G)$ cannot satisfy (d) of Definition \ref{d1} by the previous paragraph. The pairs $(X,G)$ satisfying (b) of Definition \ref{d1} with $X = \mathcal{S}_t$ are completely described in \cite[Theorem 6.2.1 (1)]{Sheikhthesis}.

Thus in order to classify all primitive permutation groups $G$ acting on finite sets $X$ in standard $t$-actions for integers $t$ such that $\mathrm{diam}(X,G) \leq 2$, we may suppose that $\mathrm{diam}(X,G) = 2$ and that (b) or (c) of Definition \ref{d1} is satisfied with the assumption that $X \not= \mathcal{S}_t$ provided that (b) holds. 

Our second main theorem is a partial classification of primitive permutation groups $G$ acting on finite sets $X$ in standard $t$-actions for integers $t$ such that $\mathrm{diam}(X,G) \leq  2$. 
\begin{theorem}\label{thm2}
Let $G$ be a primitive permutation group acting on a finite set $X$ such that $G$ has a standard $t$-action. Let $k = \min \{ t, n-t \}$. Suppose that $n \geq 8$. Let $\mathrm{diam}(X,G) = 2$ and that (b) or (c) of Definition \ref{d1} is satisfied with the assumption that $X \not= \mathcal{S}_t$ provided that (b) holds. One of the following holds. 
\begin{enumerate}
    \item $(X,G)$ is as in (b) of Definition \ref{d1} and $k=1$. 
    \item $(X,G)$ is as in (b) of Definition \ref{d1} and $k=2$, moreover $G_0 = \mathrm{P\Omega}_n^\epsilon(q)$ with $q \equiv 3 \pmod 4$ and the elements of $X$ are of type $\mathrm{O}_2^-.$
\end{enumerate}
Conversely, if (1) holds, then $\mathrm{diam}(X,G) = 2$.
\end{theorem}
\begin{remark}\rm
In part (2) of Theorem \ref{thm2} we have not been able to determine whether the orbital diameter is $2$. We
conjecture that it is at least $3$.
\end{remark}
\section{Background}

In this section we recall the background and notation which we will need.

Let us start with forms. Let $V$ be a finite vector space over a finite field $F$. Let $f$ be a map from $V \times V$ to $F$. The map $f$ is called {\it non-degenerate} if for each non-zero vector $v$ in $V$ the maps from $V$ to $F$ given by $x \to f(x,v)$ and $x \to f(v,x)$ are non-zero. A quadratic form $Q$ is called {\it non-degenerate} if its associated (symmetric) bilinear form $f$ is non-degenerate. The map $f$ is called {\it symmetric} if $f(u,v) = f(v,u)$ for all $u$, $v$ in $V$. The map $f$ is called {\it skew-symmetric} if $f(u,v) = - f(v,u)$ for all $u$, $v$ in $V$. If the characteristic of $F$ is odd and $f$ is skew-symmetric, then $f(v,v) = 0$ for all $v \in V$. The map $f$ is said to be {\it symplectic} if $f$ is skew-symmetric, bilinear, and $f(v,v) = 0$ for all $v \in V$. The map $f$ is said to be {\it unitary} if $F$ admits an involutary field automorphism $\alpha$ and $f$ is left-linear and $f(u,v) = f(v,u)^{\alpha}$ for all $u$, $v \in V$. 

Let $f$ be a bilinear or a sesquilinear form (or a unitary map) on $V$. Vectors $u$ and $v$ in $V$ are said to be {\it orthogonal} or {\it perpendicular} if $f(u,v) = 0$. For a subspace $U$ in $V$, let $U^{\perp}$ be the subspace of $V$ consisting of all vectors $v$ in $V$ such that $f(u,v) = 0$ for all $u \in U$. The form $f$ is called {\it non-singular} if $V^{\perp} = \{ 0 \}$ and {\it singular} otherwise. A subspace $U$ of $V$ is called {\it totally singular} if $f(u,v) = 0$ for all $u$, $v \in U$ and is called {\it non-degenerate} if $U \cap U^{\perp} = \{ 0 \}$.

Let $n$ denote the dimension of $V$ over $F$. Let $f$ be a non-degenerate unitary form on $V$. In this case $V$ admits an orthonormal basis (see \cite[Proposition 2.3.1]{KL}), that is, a basis $\{ e_{1}, \ldots , e_{n} \}$ such that $f(e_{i},e_{i}) = 1$ for all $i$ with $1 \leq i \leq n$ and $f(e_{i},e_{j}) = 0$ whenever $i$ and $j$ are different indices between $1$ and $n$. There is another useful basis called the {\it unitary basis} $\mathcal{B}$ for $V$. If $n = 2m$ is even, then $V$ has a basis $\{ e_{1}, \ldots , e_{m}, f_{1}, \ldots , f_{m}  \}$ and if $n = 2m+1$ is odd, then $V$ has a basis $\{ e_{1}, \ldots , e_{m}, f_{1}, \ldots , f_{m}, x  \}$ such that $f(u,v) = 0$ whenever $u$, $v \in \mathcal{B}$, except when $(u,v) = (x,x)$ or $\{ u, v \} = \{ e_{i}, f_{i} \}$ for some index $i$ with $1 \leq i \leq n$, in which case $f(u,v) = 1$. See \cite[Proposition 2.3.2]{KL}. Let $f$ be a non-degenerate symplectic form on $V$. In this case $n = 2m$ is even and $V$ admits a so-called {\it symplectic basis} $\mathcal{B} = \{ e_{1}, \ldots , e_{m}, f_{1}, \ldots , f_{m} \}$ such that $f(u,v) = 0$ for all $u$, $v \in \mathcal{B}$ unless $\{ u, v \} = \{ e_{i}, f_{i} \}$ for some $i$ with $1 \leq i \leq m$ in which case $f(e_{i},f_{i}) = - f(f_{i},e_{i}) = 1$. See \cite[Proposition 2.4.1]{KL}. Let $Q$ be a non-degenerate quadratic form on $V$. Let $f$ be the associated bilinear form. This is a non-degenerate symmetric bilinear form on $V$. In this case $V$ admits a {\it standard basis} $\mathcal{B}$ defined in the following way (see \cite[Proposition 2.5.3]{KL}). First let $n = 2m+1$ be odd. In this case $\mathcal{B} = \{ e_{1}, \ldots , e_{m}, f_{1}, \ldots , f_{m}, x \}$ such that $Q(e_{i}) = Q(f_{i}) = 0$ for all $i$ with $1 \leq i \leq m$, the vector $x$ is non-singular, and $f(u,v) = 0$ for all $u$, $v \in \mathcal{B}$ except when $\{ u, v \} = \{ e_{i}, f_{i} \}$ for some index $i$ with $1 \leq i \leq m$, in which case $f(u,v) = 1$. When $n = 2m$ is even, then there are two cases. In the first case the basis $\mathcal{B}$ is $\{ e_{1}, \ldots, e_{m}, f_{1}, \ldots , f_{m} \}$ such that $Q(e_{i}) = Q(f_{i}) = 0$ for all $i$ with $1 \leq i \leq m$ and $f(u,v) = 0$ for all $u$, $v \in \mathcal{B}$ unless $\{ u, v \} = \{ e_{i}, f_{i} \}$ for some $i$ in which case $f(u,v) = 1$. We will refer to $\langle e_{1}, \ldots, e_{m}, f_{1}, \ldots , f_{m} \rangle$ as type $\mathrm{O}_n^+.$ In the second case $\mathcal{B}$ has the form $\{ e_{1}, \ldots , e_{m-1}, f_{1}, \ldots , f_{m-1}, x, y \}$ with the following property. We have $Q(e_{i}) = Q(f_{i}) = 0$ for all $i$ with $1 \leq i \leq m$, $Q(x) = 1$, $Q(y) = \zeta$ where $\zeta$ is such that the polynomial $t^{2} + t + \zeta$ (in the variable $t$) is irreducible over $F$, and $f(u,v) = 0$ for all $u$, $v \in \mathcal{B}$ unless $\{ u, v \} = \{ e_{i}, f_{i} \}$ for some $i$ or $\{ u, v \} = \{ x, y \}$, in which case $f(u,v) = 1$.   We will refer to $\langle e_{1}, \ldots, e_{m-1}, f_{1}, \ldots , f_{m-1},x,y \rangle$ as type $\mathrm{O}_n^-.$






We recall Witt's Lemma from \cite[Proposition 2.1.6]{KL}.

\begin{theorem}[Witt's lemma]
\label{Witt}	
Assume that $(V_{1},\kappa_{1})$, $(V_{2},\kappa_{2})$ are two isometric classical geometries and that $W_{i}$ is a subspace of $V_{i}$ for $i \in \{ 1, 2 \}$. Further assume that there is an isometry $\varphi$ from $(W_{1},\kappa_{1})$ to $(W_{2},\kappa_{2})$. Then $\varphi$ extends to an isometry from $(V_{1},\kappa_{1})$ to $(V_{2},\kappa_{2})$. 
\end{theorem}

\section{Proof of Theorem \ref{thm1}}

In this section we prove Theorem \ref{thm1}. 

\subsection{Case (a)} Let $(G,X)$ satisfy (a) of Definition \ref{d1}. 
Recall that in case (a) $G_{0} = \mathrm{A}_{n}$ and $X = I^{\{ t \}}$, the set of $t$-subsets of $I = \{ 1, \ldots, n \}$ with the natural action of $\mathrm{A}_n$.  In this case 
we have that $\mathrm{diam}(X,G) = k$ by \cite[Theorem 1.1]{Sheikh} provided that $k < n/2$. Note that the fact that $k$ is a lower bound for $\mathrm{diam}(X,G)$ can be found in \cite[p. 229]{LMT}. 

\subsection{Case (d)} Let $(G,X)$ satisfy (d) of Definition \ref{d1}. Recall that in this case $G_{0} = \mathrm{Sp}_{n}(q)$, $q$ even, and a point stabilizer in $G_{0}$ is $\mathrm{O}^{\pm}_{n}(q).$ Here $\mathrm{diam}(X,G)\geq 1$ is obvious. 

In proving Theorem \ref{thm1} in cases (b) and (c) of Definition \ref{d1} we will need a lemma.  

Let $X$ be a set of vector spaces each of dimension $k$. Let $G$ be a primitive permutation group acting on $X$ in standard $t$-action. Let $U$, $U' \in X$ be in the same $G$-orbit such that $U' \not= U$ and $\dim(U \cap U')$ is maximal. Let $\mathcal{O}$ be the orbital graph containing the edge $\{ U,U' \}$. For two elements $A$ and $B$ in $X$ let $d(A,B)$ denote the distance between $A$ and $B$ in the graph $\mathcal{O}$.

\begin{lemma}\label{seged}
Let $\dim(U\cap U')=k-l$ with $l\geq 1$. If there exists $U''$ in $X$ such that $\dim(U\cap U'')=k-rl$ for some positive integer $r$, then $d(U,U'')\geq r$. In particular, if $\dim(U\cap U')=k-1$ and there is $U''$ such that $\dim(U\cap U'')=0$, then the diameter of $\mathcal{O}$ is at least $k$.  
\end{lemma}
\begin{proof}
 Let $A,B\in X$. We claim that if $d(A,B)< r$, then $\dim (A\cap B)> k-rl$. Hence if $\dim(A\cap B)\leq k-rl$, then $d(A,B)\geq r$. We will prove the claim by induction on $r$. The base case $r=1$ is clear. For $r = 2$ we may suppose that $d(A,B) =1$ and so $\dim(A \cap B) = k-l > k-2l$.
 
Suppose that $r \geq 3$ and the claim is true for $r-1$, that is, $d(A,B)<r-1$ implies that $\dim (A\cap B)> k-(r-1)l$. We wish to prove it for $r$. For this let $A$ and $B$ be members of $X$ such that $d(A,B) = r-1$. The vector space $B$ has a neighbor $C$ such that $d(A,C) < r-1$. We have $\dim (B \cap C) = k - l$ and $\dim (A\cap C)> k-(r-1)l$ by the induction hypothesis. It follows that $$k=\dim (C) \geq  \dim ((A\cap C)+(B \cap C))\geq$$ $$\geq \dim(A\cap C)+\dim(B \cap C)-\dim (A \cap B) > 2k-rl-\dim (A \cap B),$$ giving $\dim (A \cap B) > k-rl$. 
\end{proof}

Now we prove Theorem \ref{thm1} for the cases (b) and (c) of Definition \ref{d1}.


\subsection{Case (b)} Let $(G,X)$ satisfy (b) of Definition \ref{d1}. Recall that in this case $G_{0} = \mathrm{Cl}_{n}(q)$ and $X$ is an orbit of subspaces of dimension or codimension $t$ in the natural module $V_{n}(q)$; the subspaces are arbitrary if $G_{0} = \mathrm{PSL}_{n}(q)$, and otherwise are totally singular, non-degenerate, or, if $G_{0}$ is orthogonal and $q$ is even, are non-singular $1$-spaces (in which case $t=1$). \par Note that for $k=1$ the bound immediately follows so we can assume that $t\geq 2.$ 
\par 
	 Let $U$ be a vector space in $X$. We may assume in all cases that $\dim(U) \leq n/2$. For $G_{0} = \mathrm{PSL}_{n}(q)$, this follows by the use of the inverse transpose automorphism. Let $G_{0} \not= \mathrm{PSL}_{n}(q)$. For $U$ totally singular, this follows by \cite[Corollary 2.1.7]{KL} and for $U$ non-degenerate, by \cite[Lemma 2.1.5 (ii), (iii), (v)]{KL}. 

\par We begin with the case when $G_0 = \mathrm{PSL}_n(q).$

\begin{lemma}
	\label{seged2}
 Let $G_{0} = \mathrm{PSL}_{n}(q)$ and $X$ the set of all subspaces of $V_{n}(q)$ of dimension $k$. In this case $\dim(U \cap U') = k-1$ and the diameter of $\mathcal{O}$ is $k$.
\end{lemma}

\begin{proof}
	We will prove that for $G_{0} = \mathrm{PSL}_{n}(q)$ and $X$ the set of all subspaces of $V_{n}(q)$ of dimension $k$ we have $\dim (A \cap B) = r$ if and only if $d(A,B) = k-r$ for any $A$, $B$ in $X$. We proceed by induction on $r$. This is clear for $r = k$ and it is also clear by transitivity for $r = k-1$. Assume that $r \leq k-2$ and that the claim is true for larger values of $r$.
	
	Let $\dim(A \cap B) = r$. We have $d(A,B) \geq k-r$ by induction. Let $\{ e_{1}, \ldots , e_{k}\}$ be a basis for $A$ and $\{ e_{1}, \ldots, e_{r}, f_{r+1}, \ldots , f_{k} \}$ be a basis for $B$. For $i$ with $1 \leq i \leq k-r$ put $A_{i} = \langle e_{1}, \ldots , e_{r}, e_{r+1}, \ldots , f_{k - (i-1)}, \ldots , f_{k} \rangle$. It is clear that $$A, A_{1}, \ldots , A_{k-r} = B$$ is a path in $\mathcal{O}$ of length $k-r$. Thus $d(A,B) \leq k-r$ and so $d(A,B) = k-r$.
	
	Assume that $d(A,B) = k-r$. We have $\dim(A \cap B) \leq r$ by induction. There exists a vector space $C$ in $X$ with $d(A,C) = k-r-1$ and $d(C,B) = 1$. We have $\dim (A \cap C) = r+1$ and $\dim (B \cap C) = k-1$ by induction. Now
	$$k = \dim (C) \geq \dim ((A \cap C) + (B \cap C)) = \dim (A \cap C) + \dim (B \cap C) - \dim(A \cap B \cap C) =$$ $$= k+r - \dim(A \cap B \cap C).$$
	It follows that $r \geq \dim(A \cap B) \geq \dim(A \cap B \cap C) \geq r$.
\end{proof}

In fact, if $G_{0} = \mathrm{PSL}_{n}(q)$ and $X$ is the set of all $k$-dimensional subspaces of $V_{n}(q)$, then $\mathrm{diam}(X,G) = k$, provided that $k < n/2$. 

\begin{lemma}
\label{pslb}
If $G_{0} = \mathrm{PSL}_{n}(q)$ and $X$ is the set of all $k$-dimensional subspaces of $V_{n}(q)$, then $\mathrm{diam}(X,G) = k$, provided that $k < n/2$.
\end{lemma}

\begin{proof}
The lower bound $\mathrm{diam}(X,G) \geq k$ follows from Lemma \ref{seged2}. The upper bound can be seen as follows. Let $\mathcal{G}$ be an arbitrary (non-diagonal) orbital graph containing the edge $\{ W, W' \}$. Let $\mathcal{B} = \{ v_{1}, \ldots , v_{n} \}$ be a basis for $V_{n}(q)$. Since $k \leq n/2$, by applying an element of $G_{0}$ to $W$ and $W'$ if necessary, we may assume that both $W$ and $W'$ have a basis, say $\mathcal{W}$ and $\mathcal{W'}$ respectively, which are subsets of $\mathcal{B}$. Consider the natural action of $\mathrm{A}_n$ on $\mathcal{B}$. (This group may be viewed as a subgroup of $G_0$.) This defines an action of $\mathrm{A}_n$ on the set of $k$-element subsets of $\mathcal{B}$. Consider the orbital containing the edge $\{\mathcal{W}, \mathcal{W'}  \}$. This has diameter $k$ by \cite[Proposition 3.1]{Sheikh}, provided that $k < n/2$. Now let $U$ and $U'$ be arbitrary vertices in $X$ of maximal distance apart in $\mathcal{G}$. By applying an element of $G_0$ if necessary, it may be assumed that the bases $\mathcal{U}$ and $\mathcal{U'}$ of $U$ and $U'$ respectively are subsets of $\mathcal{B}$. Since $\mathrm{diam}(\mathcal{B},\mathrm{A}_n) = k$, provided that $k < n/2$, the distance between  
$\mathcal{U}$ and $\mathcal{U'}$ and thus the distance between $U$ and $U'$ in $\mathcal{G}$ is at most $k$, provided that $k < n/2$. 
\end{proof} 
The corresponding exact diameter for classical groups with socle a simple symplectic group is not known. See \cite[Conjecture 6.2.3]{Sheikhthesis}.   

Let $\mathcal{S}$ be the set of all totally singular or non-degenerate subspaces of $V_{n}(q)$ which are isomorphic to $U$.

\begin{lemma}
	\label{orbit}
	The orbit $X$ is equal to $\mathcal{S}$, unless possibly if $n$ is even and any of the following holds:
	\begin{enumerate}
		\item $G_{0} = \mathrm{P \Omega}^{+}_{n}(q)$ and $U$ is totally singular of dimension $n/2$;
		
		\item $G_{0} = \mathrm{P \Omega}^{\pm}_{n}(q)$, $q$ is odd and $U$ is non-degenerate with $\dim U$ odd.
	\end{enumerate}
	In any case $\mathcal{S}$ is the union of at most two $G$-orbits.
\end{lemma}

\begin{proof}
	This follows from Witt's lemma (see Theorem \ref{Witt}) and Propositions 4.1.3, 4.1.4, 4.1.6, 4.1.18, 4.1.19, 4.1.20 of \cite{KL}. For the definition of $c$ in the statements of these propositions see \cite[Section 3.2]{KL}.
\end{proof}

\begin{lemma}
	\label{k/2}
	If (1) of Lemma \ref{orbit} holds, then $\mathrm{diam}(X,G)=[k/2].$
\end{lemma}

\begin{proof}
First we show that $\mathrm{diam}(X,G)\geq [k/2].$ It is sufficient to prove that $G$ has an orbital graph whose diameter is at least $[k/2].$  Recall
	 that $n$ is even and $k = n/2$. Let $Q$ be the quadratic form on $V_{n}(q)$ and let $\langle , \rangle$ be the associated bilinear form. Let $\{ e_{1}, \ldots , e_{k}, f_{1}, \ldots , f_{k} \}$ be the standard basis of $V_{n}(q)$ with $Q(e_{i}) = Q(f_{i}) = 0$ and $( e_{i}, f_{j} ) = \delta_{ij}$ for all $i$ and $j$ with $1 \leq i, j \leq k$. We claim that there are vector spaces $W$, $W'$, $W''$ in $\mathcal{S}$ such that $\dim(W \cap W')$ is equal to $0$ if $k$ is even, is $1$ if $k$ is odd and $\dim(W \cap W'') = k-2$. Let $W = \langle e_{1}, \ldots , e_{k} \rangle$. Take $W'$ to be $\langle f_{1}, \ldots , f_{k} \rangle$ if $k$ is even and $W' = \langle e_{1}, f_{2}, \ldots , f_{k} \rangle$ if $k$ is odd. Put $W'' = \langle e_{1}, e_{2}, \ldots , f_{k-1}, f_{k} \rangle$.
	
	The set $\mathcal{S}$ is the union of two $\mathrm{\Omega}^{+}_{n}(q)$-orbits by Description 4 on page 30 of \cite{KL}. Let these be $\mathcal{U}^{1}_{k}$ and $\mathcal{U}^{2}_{k}$. For each $i \in \{ 1,2 \}$ vector spaces $A$ and $B$ are in $\mathcal{U}^{i}_{k}$ if and only if $k - \dim (A \cap B)$ is even by \cite[Description 4]{KL}. It follows that $W$, $W'$, $W''$ belong to the same $\mathrm{\Omega}^{+}_{n}(q)$-orbit.
	
	Let $\mathcal{O}'$ be the orbital graph containing the edge $\{ W, W'' \}$ for the permutation group $\mathrm{PSL}_{n}(q)$ acting on $\mathcal{S}$. By Lemma \ref{seged}, the distance between $W$ and $W'$ in the graph $\mathcal{O}'$ and thus also the distance in the orbital graph for $G$ containing the edge $\{  W, W'' \}$ is at least $[k/2]$.

 Now we prove that $\mathrm{diam}(X,G)\leq [k/2].$ By \cite[Lemma 6.4.2.(3)]{Sheikhthesis}, the rank of this action is $[k/2]+1$ so the upper bound on the orbital diameter follows immediately. 
\end{proof}

\begin{lemma}
	If (2) of Lemma \ref{orbit} holds, then the diameter of $\mathcal{O}$ is at least $k$.
\end{lemma}

\begin{proof}
	The set $\mathcal{S}$ of all subspaces isomorphic to $U$ consists of at most two $G_{0}$-orbits, one of these being $X$. By inspecting the non-singular symmetric bilinear form on $V_{n}(q)$ with $n \geq 3$ (see \cite[Section 3.4.6]{Wilson}), we see that there exist three subspaces $A_{1}$, $A_{2}$, $A_{3}$ in $\mathcal{S}$ such that $\dim(A_{i} \cap A_{j}) = k-1$ for all $i$ and $j$ with $1 \leq i < j \leq 3$. It follows that there is $A' \in X$ with $\dim (U \cap A') = k-1$. Similarly, one can show that there are subspaces $B_{1}$, $B_{2}$, $B_{3}$ in $\mathcal{S}$ such that $\dim (B_{i} \cap B_{j}) = 0$ for every $i$ and $j$ with $1 \leq i < j \leq 3$. It follows that there is $B' \in X$ such that $\dim (U \cap B') = 0$. Thus the diameter of $\mathcal{O}$ is at least $k$ by Lemma \ref{seged}.
\end{proof}

From now on we assume that $X = \mathcal{S}$.

\begin{lemma}
	If $X = \mathcal{S}$, then the diameter of $\mathcal{O}$ is at least $k$.
\end{lemma}

\begin{proof}
	It is sufficient to show by Lemma \ref{seged} that there are $U'$, $U'' \in \mathcal{S}$ such that $\dim(U \cap U') = k-1$ and $\dim(U \cap U'') = 0$. This is clear in case $G_{0} = \mathrm{PSL}_{n}(q)$.

 Now consider the case when $G_0$ is not orthogonal.
	Assume that $V_{n}(q)$ admits a non-singular alternating bilinear form. Let $$\{ e_{1}, \ldots , e_{m}, f_{1}, \ldots , f_{m} \}$$ be a standard symplectic basis with $n = 2m$. If $U = \langle e_{1}, \ldots , e_{\ell}, f_{1}, \ldots , f_{\ell} \rangle$ is a non-degenerate space with $k = 2 \ell$, then put
	$U' = \langle e_{1}, \ldots , e_{\ell}, f_{1}, \ldots , f_{\ell}+f_{\ell+1} \rangle$ and $U'' = \langle e_{\ell+1}, \ldots , e_{2 \ell}, f_{\ell + 1}, \ldots , f_{2 \ell} \rangle$. If $U = \langle e_{1}, \ldots , e_{k} \rangle$ is totally singular, then put $U' = \langle e_{1}, \ldots , e_{k-1}, f_{k} \rangle$ and $U'' = \langle f_{1}, \ldots , f_{k} \rangle$.
	
	Let $V_{n}(q)$ admit a non-singular conjugate-symmetric sesquilinear form and let $\{ e_{1}, \ldots , e_{n} \}$ be an orthonormal basis for this form. If $U = \langle e_{1}, \ldots , e_{k} \rangle$ is non-degenerate, then put $U' = \langle e_{1}, \ldots e_{k-1}, e_{k+1} \rangle$ and $U'' = \langle e_{k+1}, \ldots , e_{2k} \rangle$. If $U$ is totally singular, then we may write $V_{n}(q)$ as a perpendicular direct sum of non-singular $2$-spaces (together with a $1$-space in case $n$ is odd) such that each $2$-space admits a symplectic basis (see \cite[p. 67]{Wilson}). Now proceed as in the second part of the previous paragraph.

 Now consider the case when $G_0$ is orthogonal. 

Consider the case when $U$ is non-degenerate. Recall the standard basis of $V_n(q).$ This is $\{e_1,\dots,e_m,f_1,\dots,f_m\}$ or $\{e_1,\dots,e_{m-1},f_1,\dots,f_{m-1},x,y\}$ for $n=2m$ of type $\mathrm{O}_n^+$ or $\mathrm{O}_n^-$ respectively, and for $n=2m+1$ it is  $\{e_1,\dots,e_m,f_1,\dots,f_m\,x\}.$ For all $i,j$ we have $Q(e_i)=Q(f_i)=0,$ $( e_i,e_j)=( f_i,f_j) =0,$ $( e_i,f_j)=\delta_{i,j}$ and $( x,y) =1,$ $( e_i,x) =( f_i,x) =( e_i,y)=( f_i,y) =0,$ $Q(x)=1$ and $Q(y)=\alpha$  as in \cite[Proposition 2.5.3]{KL}. Now we will find $U$, $U'$, $U''$ such that $\dim(U\cap U')=k-1$ and $\dim (U\cap U'')=0,$ and then the result will follow by Lemma \ref{seged}. 
We may assume that $k\geq 3$ and so $n\geq 6.$ We have several cases to consider. 
Assume that $n$ is even.

Assume $k=2l.$

Assume $U$ is of type $\mathrm{O}_{2l}^+.$ For $k< n/2$ or $G_0 = \mathrm{P\Omega}_n^+(q),$ choose $$U=\langle e_1,\dots,e_l,f_1,\dots,f_l\rangle,$$  $$U'=\langle e_1,\dots,e_l,f_1,\dots,f_{l-1},f_l+f_{l+1}\rangle$$ and $$U''=\langle e_{l+1},\dots,e_{2l},f_{l+1},\dots,f_{2l}\rangle.$$    
For $k= n/2$ or $G_0= \mathrm{P\Omega}_n^-(q),$ choose $$U=\langle e_1,\dots,e_l,f_1,\dots,f_l\rangle,$$  $$U'=\langle e_1,\dots,e_l,f_1,\dots,f_{l-1},e_l- f_l+x\rangle$$ and $$U''=\langle e_{l+1},\dots,e_{2l-1},f_{l+1},\dots,f_{2l-1},e_1-f_1+x,e_1-\alpha f_1+y\rangle. $$

Assume $U$ is of type $\mathrm{O}_{2l}^-.$ For $G_0=\mathrm{P\Omega}_n^-(q),$ choose $$U=\langle e_1,\dots,e_{l-1},f_1,\dots,f_{l-1},f_l+\alpha e_l,f_{l+1}+x+e_l\rangle,$$  $$U'=\langle e_1,\dots,e_{l-1},f_1,\dots,f_{l-1},f_l+\alpha e_l,f_{l+1}+e_{l+1}+e_l\rangle$$ and $$U''=\langle e_{l+1},\dots,e_{2l-1},f_{l+1},\dots,f_{2l-1},y,e_1+x\rangle.$$  
For  $G_0=\mathrm{P\Omega}_n^+(q),$ choose $$U=\langle e_1,\dots,e_{l-1},f_1,\dots,f_{l-1},f_l+\alpha e_l,f_{l+1}+e_{l+1}+e_l\rangle,$$  $$U'=\langle e_1,\dots,e_{l-1},f_1,\dots,f_{l-1},f_l+\alpha e_l,f_{l+2}+e_{l+2}+e_l\rangle$$ and $$U''=\langle e_{l+2},\dots,e_{2l},f_{l+2},\dots,f_{2l},e_1+\alpha f_1+f_l,e_2+ f_2+f_{l+1}+f_1\rangle.$$

Assume $k=2l+1.$ For $G_0 = \mathrm{P\Omega}_n^-(q),$ choose $$U=\langle e_1,\dots,e_{l},f_1,\dots,f_{l},x\rangle,$$  $$U=\langle e_1,\dots,e_{l},f_1,\dots,f_{l},f_{l+1}+e_{l+1}\rangle,$$  and $$U''=\langle e_{l+1},\dots,e_{2l},f_{l+1},\dots,f_{2l},y\rangle.$$ 
For  $G_0 = \mathrm{P\Omega}_n^+(q),$ choose $$U=\langle e_1,\dots,e_{l},f_1,\dots,f_{l},f_{2l+1}+e_{2l+1}+e_{l+1}\rangle,$$  $$U'=\langle e_1,\dots,e_{l},f_1,\dots,f_{l},f_{l+2}+e_{l+2}\rangle$$ and $$U''=\langle e_{l+2},\dots,e_{2l+1},f_{l+2},\dots,f_{2l+1},e_1+ f_1+f_{l+1}\rangle.$$

Now assume $n$ is odd. Now $k \leq (n-1)/2.$ We can choose the same $U,$ $U'$ and $U''$ as we did for $G_0 = \mathrm{P\Omega}_{n-1}^+(q).$

Now consider the case when $U$ is totally singular. 

Recall that there is a basis as in \cite[Proposition 2.5.3]{KL}. In the various cases the maximum possible value of $k = \dim(U)$ is determined according to \cite[Proposition 2.5.4]{KL}. Using this information one can see that there are subspaces $U'$ and $U''$ in $X$ with $\dim(U \cap U') = k-1$ and $\dim(U \cap U'') = 0$. It follows by Lemma \ref{seged} that in this case the diameter of $\mathcal{O}$ is at least $k$. 
\end{proof}

\subsection{Case (c)} Let $(G,X)$ satisfy (c) of Definition \ref{d1}. Recall that in this case $G_{0} = \mathrm{PSL}_{n}(q)$, $G$ contains a graph automorphism of $G_{0}$, and $X$ is an orbit of pairs of subspaces $\{ U, W \}$ of $V = V_{n}(q)$, where either $U \subseteq W$ or $V = U \oplus W$, and $\dim U = t$, $\dim W = n-t$. We may suppose that $k = t \leq n/2$. Let $\{ e_{1}, \ldots , e_{n} \}$ be a basis for $V = V_{n}(q)$ and let $\sigma$ be the inverse transpose automorphism with respect to this basis. We may view $\sigma$ also as an automorphism of $G_{0}$ and thus an element of $G$. Since we aim to establish a lower bound for $\mathrm{diam}(X,G)$, there is no harm in assuming that $G$ is as large as possible, $\mathrm{P\Gamma L}(V) \leq G$ and $\sigma \in G$. Put $U = \langle e_{1}, \ldots , e_{t} \rangle$ and define the subspaces $W_{1}$ and $W_{2}$ to be $\langle e_{t+1}, \ldots , e_{n} \rangle$ and $\langle e_{1}, \ldots , e_{n-t} \rangle$ respectively. Let $p_{1}$ be the permutation matrix swapping $e_{t}$ and $e_{t+1}$ and fixing all other basis vectors, and let $p_{2}$ be the permutation matrix interchanging $e_{t}$ and $e_{n-t+1}$ and fixing all other basis vectors. Let $i$ be $1$ or $2$. Observe that if $\{ U, W_{i} \} \in X$ then $\{ Up_{i}, W_{i}p_{i}  \} \in X$. Let us denote $\{ Up_{i}, W_{i}p_{i}  \}$ by $\{ U, W_{i} \}p_{i}$. Let $\mathcal{O}_{i}$ be the orbital graph of $G$ containing the edge $E_{i} = \{ \{ U, W_{i} \}, \{ U, W_{i} \}p_{i} \}$. Let the image of $\{ U, W_{i} \}$ under $\sigma$ be denoted by $\{ U, W_{i} \} \sigma$. This is $\{ U \sigma, W_{i}\sigma \}$ where $U \sigma = W_{1}$, $W_{1}\sigma = U$ and $W_{2}\sigma = \langle e_{n-t+1}, \ldots , e_{n} \rangle$. Observe that $\sigma$ centralizes $p_{i}$. It follows that the edge $E_{1}$ is fixed by $\sigma$. Let $h$ be the permutation matrix swapping $e_{j}$ with $e_{n+1-j}$ for every $j$ in $\{ 1, \ldots , t \}$ and fixing the vectors $e_{t+1}, \ldots , e_{n-t}$. Observe that $\sigma h \in G$ and $\{ U, W_{2} \} \sigma h = \{ U \sigma, W_{2} \sigma \} h = \{ W_{1}, \langle e_{n-t+1}, \ldots , e_{n} \rangle \} h = \{ W_{2}, U \}$. Observe also that $h$ centralizes $p_{2}$. It follows that $\sigma h$ fixes the edge $E_{2}$. To summarize, $\mathcal{O}_{i}$ is an orbital graph also of $\mathrm{P\Gamma L}(V)$. From now on assume that $G = \mathrm{P\Gamma L}(V)$. 

Let $t < n/2$. For each pair $\{ U', W \}$ in $X$ one of $\dim (U')$ and $\dim(W)$ is smaller, and so, by disregarding the vector space of larger dimension from every pair in $X$, we may repeat the argument in Lemma \ref{seged} and obtain that the diameter of $\mathcal{O}_{i}$ is at least $k$. Thus $\mathrm{diam}(X,G) \geq k$.

Let $t = n/2$. The case when the vector spaces are equal in any pair in $X$ was treated in (b). Assume that this is not the case and that
$V = U \oplus W$ whenever $\{ U, W \} \in X$. Since $G_{0}$ is contained in $G$, the orbit $X$ is precisely the set of all pairs $\{ U , W \}$ such that $V = U \oplus W$ and $\dim(U) = \dim(W) = t$. Consider the orbital graph $\mathcal{O}_{1}$. If $\{ \{ U_{1}, W_{1} \}, \{ U_{2}, W_{2} \} \}$ is an edge in $\mathcal{O}_{1}$, then $\dim (U_{1} \cap U_{2}) = t-1$, $\dim (W_{1} \cap W_{2}) = t-1$, and $\dim (U_{1} \cap W_{2}) = \dim (U_{2} \cap W_{1}) = 1$ (with possibly $U_1$ and $W_1$ interchanged).

For arbitrary pairs $A = \{ U_{1}, W_{1} \}$ and $B = \{ U_{2}, W_{2} \}$ in $X$, let $d(A,B)$ denote their distance in $\mathcal{O}_{1}$ and define $\dim(A,B)$ to be $$\max \{ \dim (U_{1} \cap W_{2}), \dim (U_{1} \cap U_{2}), \dim (U_{2} \cap W_{1}), \dim (W_{1}\cap W_{2})\}.$$

\begin{lemma}
	\label{pairofspaces}
	For pairs $A$ and $B$ in $X$ we have $d(A,B) \geq t - \dim(A,B)$.
\end{lemma}

\begin{proof}
	Let $A = \{ U_{1}, W_{1} \}$ and $B = \{ U_{2}, W_{2} \}$. We proceed by induction on $d(A,B)$. If $d(A,B) = 1$, then $\dim(A,B) = t-1$. Assume that $d(A,B) = d \geq 2$. If $\dim(A,B) \geq t-d$, then the claim holds. Assume that $\dim(A,B) \leq t-d-1$. There exists $C = \{ U_{3}, W_{3} \} \in X$ such that $d(A,C) = 1$ and $d(C,B) \leq d-1$. There exist an element of $A$ and an element of $C$ whose intersection has dimension $t-1$, by the definition of $\mathcal{O}_{1}$. Let these be $U_{1}$ and $U_{3}$. There exists, by induction, an element, say $U_{2}$, of $B$ whose intersection with $U_{3}$ has dimension at least $t-d+1$. Since $\dim(A,B) \leq t-d-1$, we have $\dim(U_{1} \cap U_{2}) \leq t-d-1$. Since $(U_{1} \cap U_{3}) + (U_{2} \cap U_{3})$ is contained in $U_{3}$, it follows that $\dim((U_{1} \cap U_{3}) + (U_{2} \cap U_{3})) \leq \dim(U_{3}) = t$. But this leads to a contradiction since the left-hand side of this inequality is at least $t+1$ since $$t+1 = (t-1) + (t-d+1) - (t-d-1) \leq \dim(U_{1} \cap U_{3}) + \dim(U_{2} \cap U_{3}) - \dim(U_{1} \cap U_{2} \cap U_{3}).$$
\end{proof}

The invariant $\dim(A,B)$ may take the value $0$ as shown by the following example. Let $A = \{ U_{1}, W_{1} \}$ be a pair in $X$. Let $\alpha$ be an isomorphism between $U_{1}$ and $W_{1}$ and let $\beta$ be such an isomorphism of $W_{1}$ which fixes no non-zero vector in $W_{1}$. Define $U_{2} = \{ u_{1} + u_{1}\alpha \mid u_{1} \in U_{1} \}$ and $W_{2} = \{ u_{1} + u_{1}\alpha \beta \mid u_{1} \in U_{1} \}$. By transitivity, the pair $B = \{ U_{2}, W_{2} \}$ is also in $X$, and, by construction, $\dim(A,B) = 0$.

We now have $\mathrm{diam}(X,G) \geq t = k$ by Lemma \ref{pairofspaces}.

\section{Proof of Theorem \ref{thm2}}

In this section we prove Theorem \ref{thm2}.

Let $G$ be a primitive permutation group acting on a finite set $X$ such that $G$ has a standard $t$-action for some positive integer $t$. 


\subsection{Case (b)} Let $(X,G)$ be as in (b) of Definition \ref{d1}. By the beginning of Section 3.2, we may assume that $\dim(U) = k$ where $U$ is a member of $X$. Let $X \not= \mathcal{S}_t$. In particular, $G_{0} \not= \mathrm{PSL}_{n}(q)$.



We may assume that $U$ is a non-degenerate vector space. If $\dim(U) \geq 3$, then $\mathrm{diam}(X,G) \geq 3$ by Theorem \ref{thm1}. Thus, in order to prove Theorem \ref{thm2}, we assume that $k = 1$ or $k = 2$.  

We first consider the case when $\dim(U)=1$. Since $U$ is non-degenerate and the case $G_{0} = \mathrm{PSL}_{n}(q)$ was treated earlier, we can assume that $G_{0}$ is unitary or orthogonal. 

\begin{lemma}
If $G_{0}$ is a unitary group with $n \geq 5$ and $X$ is the set of all non-degenerate $1$-spaces, then $\mathrm{diam}(X,G) = 2$.
\end{lemma}

\begin{proof}
The natural module $V = V_{n}(q)$ is a unitary space with a non-degenerate unitary form $f$. Let $U$ be a member of $X$. This is a non-degenerate subspace of dimension $1$. 
In this case $X=\{\langle v\rangle \vert f(v,v)=1\}.$ Let $$S_{\langle v\rangle}=\{\alpha v\vert f(\alpha v,\alpha v)=1\}=\{\alpha v\vert \alpha^{q+1}=1 \}$$ where $\alpha\in \mathbb{F}_{q^2}.$ This set has cardinality larger than 1. Let $D=\{ \alpha\in \mathbb{F}_{q^2}\vert \alpha^{q+1}=1\}.$ This is a cyclic group of order $q+1.$ For a scalar $\lambda\in \mathbb{F}_{q^2},$ we define $$\{ \lambda, \overline{\lambda} \} D = \{ \lambda \alpha\vert \alpha \in D\} \cup \{ \overline{\lambda} \alpha\vert \alpha \in D\}$$ and let $$\Delta_{\lambda D}=\{ \{ \langle v\rangle, \langle w \rangle \} \vert f(v,v)=f(w,w)=1, f(v,w)\in \{ \lambda, \overline{\lambda} \} D\}.$$ 

The set $\Delta_{\lambda D}$ is non-empty for every $\lambda \in \mathbb{F}_{q^{2}}$ provided that $n \geq 3$. To see this, let $\mathcal{B}$ be a unitary basis for $V$ and let $\chi \in \mathbb{F}_{q^{2}}$ be such that $\chi + \overline{\chi} = 1$ ($\chi$ exists by the surjectivity of the trace map). Take $v = v_{\lambda} = e_{1} + \chi f_{1}$ where $e_{1}$, $f_{1} \in \mathcal{B}$. If $n \geq 3$ is odd, then $x \in \mathcal{B}$ and we may take $w = w_{\lambda} = x + \overline{\lambda}f_{1}$. On the other hand, if $n \geq 4$ is even, then $e_{2}, f_{2} \in \mathcal{B}$ and we take $w = w_{\lambda} = e_{2} + \chi f_{2} + \overline{\lambda}f_{1}$. 

Observe that $\Delta_{\lambda D}$ is invariant under $G.$ By applying Theorem \ref{Witt} on 2-spaces of the form $\langle v, w \rangle\,\,\text{such that}\,\, f(v,v)=f(w,w)=1, f(v,w)\in \{ \lambda, \overline{\lambda} \} D,$ it follows that $\Delta_{\lambda D}$ is an orbital. For any $u, v\in V$ such that $f(u,u)=f(v,v)=1$ we have that $f(u,v)$ is either $0$ or is in the union of one or two cosets of $D,$ hence all non-diagonal orbitals are of the form $\Delta_{\lambda D}$ for some $\lambda$. In particular, the rank of the permutation group $G$ on $X$ is at least $3$ implying that $\mathrm{diam}(X,G) \not= 1$. 

Let $\lambda$, $\mu \in \mathbb{F}_{q^{2}}$ such that $\lambda \not= \mu$. Let $\mathcal{B}$ be a unitary basis for $V$. In this case $e_{1}$, $f_{1}$, $e_{2}$, $f_{2}$ are in $\mathcal{B}$. Let $v = e_{1} + \chi f_{1}$ and $w = \mu f_{1} + e_{2} + \chi f_{2}$ for some $\chi \in \mathbb{F}_{q^{2}}$ such that $\chi + \overline{\chi} = 1$. If $n=5$, then $x \in \mathcal{B}$ and we may take $u$ to be $x + \lambda f_{1} + \lambda f_{2}$. If $n \geq 6$, then $e_{3}$, $f_{3} \in \mathcal{B}$ and we take $u$ to be $f_{3} + \chi e_{3} + \lambda f_{1} + \lambda f_{2}$. We have $f(v,v) = f(w,w) = f(u,u) = 1$ and $f(v,w) = \overline{\mu}$, $f(u,v) = \lambda$, $f(u,w) = \lambda$. 
\end{proof}

\begin{lemma}
If $G_{0}$ is an orthogonal group with $n\geq 5,$ $q$ odd and $X$ is an orbit of non-degenerate $1$-spaces, then $\mathrm{diam}(X,G) = 2$.
\end{lemma}

\begin{proof}
The natural module $V = V_{n}(q)$ is an orthogonal space with a non-degenerate quadratic form $Q$ and associated bilinear form $f$. For a non-zero vector $v \in V$ such that $Q(v)$ is a square in $\mathbb{F}_{q}$ there are two vectors $u$ in $\langle v \rangle$ such that $Q(u) = 1$. Let us denote these vectors by $v_{0}$ and $-v_{0}$.  

Observe that for each $U \in X$ there exists $u \in U$ such that $Q(u) = 1$. Let $U$ and $U'$ be two distinct vertices in $X$. Let $u$ be a generator of $U$ and $u'$ a generator of $U'$. Consider the orbital graph $\mathcal{O} = \{ U, U' \}^{G}$ and the subspace $W = \langle u, u' \rangle = \langle u_{0}, u'_{0} \rangle$. Denote $f(u_{0},u'_{0})$ by $\alpha$. Observe that the conditions $Q(u_{0}) = Q(u'_{0}) = 1$, $f(u_{0},u'_{0}) = \alpha$ uniquely determine the isomorphism type of the space $W$. It follows by Theorem \ref{Witt} that the graph $\mathcal{O}$ is uniquely determined by the isomorphism type of $W$. On the other hand, the $2$-space $W$ is uniquely determined by $\pm \alpha$. Hence all orbitals are of the form $$\Delta_{\pm \alpha}=\{ \{ \langle v\rangle, \langle w \rangle \} \vert \langle v\rangle\neq \langle w \rangle, Q(v)=Q(w)=1, f(v,w)=\pm \alpha\}.$$ 
To prove that the diameter of each associated orbital graph is at most two, we need to show that for any two non-degenerate $1$-space, there is an element of $x$ that neighbours them both in $\Delta_{\pm \lambda}$. As all pairs $\{ \langle v\rangle, \langle w \rangle \}$ such that $Q(v)=Q(w)=1$ and $f(v,w)=\alpha$ are in the same orbit, it is sufficient to show that for two specific vectors with $Q(v)=Q(w)=1$ and $f(v,w)=\alpha$ there is $x$ such that $Q(x)=1$ and $f(x,w)=\pm \lambda$ and $f(v,x)=\pm \lambda.$ Let $v=e_1+f_1$ and $w=\alpha f_1+e_2+f_2.$ For $n=5$ and $(n,\epsilon)=(6,-)$ choosing $x=\lambda f_1+\lambda f_2+x$ works and for $n\geq 7$ or $(n,\epsilon)=(6,+)$ choosing $x=\lambda f_1+\lambda f_2+e_3+f_3$ works. 
\end{proof}

The following result concerns the case when $G_{0}$ is orthogonal, $q$ is even and $X$ is the set of non-singular $1$-spaces. 
\begin{lemma}
    Let $G_0$ be orthogonal, with $n\geq 8$ even,  $q$ even and $X$ the set of non-singular $1$-spaces. Then $\mathrm{diam}(X,G) = 2.$
\end{lemma}
\begin{proof}
    The action of $G$ on $X$ is equivalent to on the set $X'$ of non-singular vectors $v$ such that $Q(v)=1.$ We can see using Witt's Lemma (see Theorem \ref{Witt}) that the orbitals are of the form $$\Delta_\alpha=\{\{a,b\}\vert a,b\in X', ( a,b)=\alpha \}$$ with $\alpha\in \mathbb{F}_q.$ Note that $Q(a+b)=( a,b).$
    Consider the orbital graph $\Delta_\lambda,$ $\lambda
\in \mathbb{F}_q.$ 

Fix $x_0,v_0 \in X'.$ Let $Q(v_0+x_0)=\sigma\in \mathbf{F}_q.$ Assume there is $w_0\in X'$ such that $Q(v_0+w_0)=Q(w_0+x_0)=\lambda.$ Let $x,v\in X'$ arbitrary such that $Q(x+v)=\sigma.$ Now there exists $g\in G$ such that $(v_0,x_0)^g=(v,x),$ so $Q(v+w_0^g)=Q(w_0^g+x)=\lambda,$ so $w_0^g$ is a common neighbour of $v$ and $x.$

For $\lambda=0$ choose $v_0=f_2+e_2+e_1$ and $x_0=f_2+e_2+\sigma f_1.$ Then for $w_0=e_2+f_2$ we have  $Q(v_0+w_0)=Q(w_0+x_0)=\lambda,$ as required.
For $\lambda\neq 0,$ choose $v_0=e_1+e_2+e_3+f_3$ and $x_0=(\sigma-\lambda)f_1+e_3+f_3+\lambda f_2.$ Then for $w_0=e_2+e_3+f_3+\lambda f_1$ we have $Q(v_0+w_0)=Q(w_0+x_0)=\lambda,$ as required.
\end{proof}




Now consider the case when $\dim(U)=2$ where $U \in X$. 

\begin{lemma}
Let $n\geq 4$. Let $G_0$ be $\mathrm{PSp}_n(q)$, $\mathrm{PSU}_n(q)$ or $\mathrm{P\Omega}_n^\epsilon(q)$ provided that, in the latter case, $(n,\epsilon)\neq (4,-)$. Let $X$ be the set of  those non-degenerate subspace of $V$ which have dimension $2$. Furthermore, assume that each subspace in $X$ is of type $\mathrm{O}_{2}^{+}$ when $G_{0} = \mathrm{P\Omega}_n^\epsilon(q)$. Then $\mathrm{diam}(X,G)\geq 3.$ 
\end{lemma}

\begin{proof}
It is sufficient to exhibit an orbital graph with diameter at least $3$. Let $\mathcal{B}$ be a symplectic, a unitary or a standard basis and let $e_{1}$, $e_{2}$, $f_{1}$, $f_{2}$ as defined in Section 2. Let $U = \langle e_{1}, f_{1} \rangle$ and $U' = \langle e_{1} + e_{2}, f_{1} \rangle$. These are vertices in $X$ with intersection $\langle f_{1} \rangle$, a singular space of dimension $1$. Consider the orbital graph $\mathcal{O}=\{U,U'\}^G$. Let $U'' =\langle e_2,f_2\rangle.$ We will show that $d(U,U'')\geq 3$. Clearly, $U$ and $U''$ are not adjacent in $\mathcal{O}$. Assume for a contradiction that there is a vertex $W$ in $X$ such that $W$ is adjacent to both $U$ and $U''$. In this case $U\cap W=\langle u \rangle$ and $U''\cap W=\langle v \rangle$ are singular $1$-spaces where $u$ and $v$ are non-zero vectors in $V$. Since $U \cap U'' = 0$, the $1$-spaces $\langle u \rangle$ and $\langle v \rangle$ are distinct and so $W = \langle u, v \rangle$. On the other hand, since $U$ and $U''$ are perpendicular, $u$ is orthogonal to $v$, and so $W$ is totally singular. This is a contradiction to the fact that $W$ is non-degenerate. 
\end{proof}

We continue this section by proving that for $G_{0} = \mathrm{P\Omega}_n^\epsilon(q)$ such that $q\equiv 1\mod 4$ and $X$ is of type $\mathrm{O}_{2}^{-},$ the orbital diameter is at least $3$.  
\begin{lemma}
   Let $n\geq 7.$ Let $G_0$ be $\mathrm{P\Omega}_n^\epsilon(q)$ such that $q\equiv 1\mod 4$ and $X$ be the set of subspaces of type $\mathrm{O}_{2}^{-}.$ Then $\mathrm{diam}(X,G)\geq 3.$ 
\end{lemma}

\begin{proof}
    It is sufficient to exhibit an orbital graph with diameter at least 3. 
    Let $e_{1}$, $e_{2}$, $e_{3},$ $e_{4},$ $f_{1}$, $f_{2},$ $f_{3}$, $f_{4}$ and $x$ as defined in Section 2, and let $\zeta\in \mathbb{F}_q$ such that $x^2+x+1$ is irreducible. Let $U = \langle e_{1}+f_{1}, e_{1}+\zeta e_{2}+f_{2} \rangle$ and $U' = \langle e_{1} + f_{1}, e_{1}+\zeta e_{3}+f_{3} \rangle$. Since both $U$ and $U'$ are spanned by two vectors whose quadratic forms are $1$ and $\zeta,$ respectively, and $(e_{1}+f_{1}, e_{1}+\zeta e_{2}+f_{2})=(e_{1} + f_{1}, e_{1}+\zeta e_{3}+f_{3})=1,$ they are of type $\mathrm{O}_2^-.$ These are vertices in $X$ with intersection $\langle e_{1}+f_{1} \rangle$, a $1$-space containing a vector with quadratic form equal to 1. Consider the orbital graph $\mathcal{O}=\{U,U'\}^G$. Let $U'' =\langle \zeta e_3+f_3,e_3+v\rangle,$ where $v=x$ for $n=7$ and $\mathrm{P\Omega}_8^-(q)$ and $v=e_4+f_4$ otherwise. We will show that $d(U,U'')\geq 3$. Clearly, $U$ and $U''$ are not adjacent in $\mathcal{O}$. Assume for a contradiction that there is a vertex $W$ in $X$ such that $W$ is adjacent to both $U$ and $U''$. In this case, without loss of generality, $U\cap W=\langle u \rangle$ and $U''\cap W=\langle w \rangle$ where $Q(u)=1$ and $Q(w)=1.$ Since $U \cap U'' = 0$, the $1$-spaces $\langle u \rangle$ and $\langle w \rangle$ are distinct and so $W = \langle u, w \rangle$. Let $\sigma\in \mathbb{F}_q$ such that $\sigma^2=-1.$ Since $q\equiv 1 \mod 4,$ we know that that $-1$ is a square, so such $\sigma$ exists. We have that $u+\sigma w\in W$ and since since $U$ and $U''$ are perpendicular, $f(u,w)=0$ and so we have $Q(u+\sigma w)=1+\sigma^2=0,$ so $W$ is of type $\mathrm{O}_2^+,$ as a 2-space of type $\mathrm{O}_2^-$ does not contain a totally singular $1$-space. This is a contradiction.  
\end{proof}


We finish the proof of Theorem \ref{thm2} by considering pairs $(X,G)$ as in (c) of Definition \ref{d1}. We show that in this case the orbital diameter is at least $3$. 

\subsection{Case (c)} Let $(X,G)$ be as in (c) of Definition \ref{d1}. In this case $G_{0} = \mathrm{PSL}_{n}(q)$, the group $G$ contains a graph automorphism $\sigma$ of $G_0$, and $X$ is an orbit of pairs $\{ U, W \}$ of subspaces of $V = V_{n}(q)$, where either $U \subseteq W$ or $V = U \oplus W$, and $\dim U = t$, $\dim W = n-t$. We may suppose that $k=t$. The possibilities are $t=1$ or $t=2$ by Theorem \ref{thm1}. Let $n \geq 8$. The automorphism $\sigma$ defines a bijection between the set of $t$-dimensional subspaces of $V$ to the set of $(n-t)$-dimensional subspaces of $V$. Let $\{ e_{1}, \ldots , e_{n} \}$ be a basis for $V$. Let $U = \langle  e_{1}, \ldots , e_{t} \rangle$. Let $W_1$ be the subspace  $\langle  e_{1}, \ldots , e_{n-t} \rangle$ in the first case and let $W_1$ be the subspace $\langle  e_{t+1}, \ldots , e_{n} \rangle$ in the second case. Let $W_2$ be the subspace  $\langle  e_{1}, \ldots , e_{n-t-1}, e_{n-t+1} \rangle$ in the first case and let $W_2$ be the subspace $\langle  e_{t+1} + e_{1}, e_{t+2}, \ldots , e_{n} \rangle$ in the second case. Consider the orbital graph $\Gamma = \{ \{  U, W_1 \}, \{ U, W_2 \} \}^{G}$. Observe that if $\{  \{ U_{1}, W_{3}  \}, \{ U_{2}, W_{4}  \} \}$ is an edge in $\Gamma$, then $\{ U_{1}, W_{3} \} \cap \{ U_{2}, W_{4} \} \not= \emptyset$. Let $\pi \in G_0$ be the element which permutes the basis vectors in such a way that $e_1$ is taken to $e_2$, $e_2$ is taken to $e_n$, $e_n$ is taken to $e_1$ and all other basis vectors are fixed. It is easy to see that $\{ U^{\pi}, W_{1}^{\pi} \} \in X$. The pair $\{   \{  U, W_1 \}, \{  U^{\pi}, W_{1}^{\pi}  \}   \}$ is not an edge in $\Gamma$ since $\{ U, W_1 \}  \cap \{  U^{\pi}, W_{1}^{\pi} \} = \emptyset$. Assume for a contradiction that the diameter of $\Gamma$ is $2$. Let $\{ U', W' \} \in X$ be a vertex adjacent to both $\{  U, W_1 \}$ and $\{  U^{\pi}, W_{1}^{\pi}  \}$ with $\dim U' = t$ and $\dim W' = n-t$. There are two cases: (i) $U' = U$ and $W' = W_{1}^{\pi}$ or (ii) $W' = W_1$ and $U' = U^{\pi}$. In both cases we arrive to the contradiction that $\{  U', W'  \}  \not\in X$.


\end{document}